\documentclass{article}
\usepackage{mdwlist, amsmath, amssymb, amsthm}
\newcommand{\supp}{\operatorname{supp}} 
\newtheorem{theorem}{Theorem}
\newtheorem*{theorem*}{Theorem}
\newtheorem{proposition}{Proposition}
\newtheorem{lemma}{Lemma}
\newtheorem{corollary}{Corollary}
\theoremstyle{definition}
\newtheorem{definition}{Definition}
\newtheorem{example}{Example}
\theoremstyle{remark}
\newtheorem{remark}{Remark}

\begin{document}

%%%%%%%%%%%%%%%%%%% Publisher's Area please ignore %%%%%%%%%%%%%%%%%%%%%%%
%\catchline{}{}{}{}{}
%%%%%%%%%%%%%%%%%%%%%%%%%%%%%%%%%%%%%%%%%%%%%%%%%%%%%%%%%%%%%%%%%%%%%%%%%%

\title{Note on limit distribution of normalized return times and escape rate}

\author{Xuan Zhang\\ Department of Mathematics, The Pennsylvania State University\\
State College, PA 16802,
USA\\
zhang\_x@math.psu.edu\footnote{Current address: Instituto de Matem\'{a}tica,
Universidade Federal do Rio de Janeiro, Rio de Janeiro, R.J., Brasil.}}
\date{September 18, 2015}

\maketitle

%\begin{history}
%\received{(Day Month Year)}
%\revised{(Day Month Year)}
%%\accepted{(Day Month Year)}
%%\comby{(xxxxxxxxxx)}
%\end{history}

\begin{abstract}
In this note we discuss limit distribution of normalized return times for shrinking targets and draw a necessary and sufficient condition using sweep-out sequence in order for the limit distribution to be exponential with parameter $1$. The normalizing coefficients are the same as sizes of the targets. Moreover we study escape rate, namely the exponential decay rate of sweep-out sequence and prove that in $\psi$-mixing systems for a certain class of sets the escape rate is in limit proportional to the size of the set.
\end{abstract}

\section{Introduction}	
Let $(X,\mathcal B, \mu)$ be a probability space and let $T:X\to X$ be a measure-preserving and ergodic transformation. For any set $A$ with a positive measure, \begin{equation}\label{def:return}\tau_A(x):=\inf\{k\ge1: T^k(x)\in A\}\end{equation} is called the (first) return time of $x$ if $x\in A$ or the (first) hitting time if $x\in X$ in general. By Poincar\' e's recurrence theorem and ergodicity, $\tau_A$ is finite almost everywhere in $X$ and the trajectory of almost every point in $X$ hits $A$ infinitely many times.  Also Kac's theorem says  that $\int_A \tau_A d\mu=1$. To understand finer statistical properties of return times, people are interested in finding laws similar to those in probability theory. We consider a sequence of positively-measured sets $\{A_n\}_{n=1}^\infty$ with $\mu(A_n)$ decreasing to $0$. Denote by $(A_n,\mathcal B_n, \mu_n:=\frac{\mu|_{A_n}}{\mu(A_n)})$ the induced system. For every $x\in A_n$, the return time of $x$ to $A_n$ is $$\tau_n(x):=\inf\{k\ge 1:T^k(x)\in A_n\}.$$ The question we attempt to answer in this note is, what is a necessary condition in order for the limit distribution of the normalized return time $\mu(A_n)\tau_n$ to be exponential with parameter $1$? The reciprocal of the normalizing factor, $1/\mu(A_n)$, is the expectation of $\tau_{A_n}$ in the induced system according to Kac's theorem. The reason why do we work on this type of distributional limit of normalized return times and why the exponential distribution is a dynamical system  transliteration of the following probability result. Given an array of independent  Bernoulli trials $\{X_{n,m}\}$ with $\mathbb P(X_{n,m}=1)=p_n$ and $\mathbb P(X_{n,m}=0)=1-p_n$ and let $\tau_n=\inf\{k: X_{n,k}=1\}$,  if $p_n(=1/\mathbb E\tau_n)\to0$ then $\lim_{n\to\infty}\mathbb P(\tau_n>[\frac{t}{p_n}])=e^{-t}$.  Historic expositions in both aspects of probability theory and dynamical system can be found in \cite{ag2001,s2009,ha2013}. There have been many works verifying exponential limit distribution for normalized return times in various specific cases, for instances \cite{p1991,h1993,hsv1999,dgs2004}. In more abstract settings, like in \cite{gs1997,ab2004,hp2014} etc., exponential limit distribution has also been demonstrated for normalized return times, whose normalizing coefficients possibly differ from sizes of the targets. On the other hand Lacroix in \cite{l2002} showed that limit distribution of $\mu(A_n)\tau_n$ can indeed be rather arbitrary when we have no good controls on the system or on shrinking targets $\{A_n\}$. In the rest of this note the normalizing coefficients of return times are always made as the sizes of the targets in accordance with Kac's theorem and exponential distribution always means the exponential distribution with parameter $1$.

It is natural to ask what should be the necessary conditions to guarantee exponential limit distribution. This direction is less fruitful, in fact the conditions in \cite{hsv1999} and \cite{hlv2005} are the only ones that are known to us so far. In \cite{hsv1999} it states that the limit distribution of the normalized return time has an exponential law with parameter $1$ if and only if the difference between distributions of the return time and of the hitting time decreases uniformly to $0$. In \cite{hlv2005} an integral equation is discovered relating the limit distribution of the normalized return time to the limit distribution of the normalized hitting time. Moreover the equation implies that one limit distribution is exponential if and only if the same holds for the other one. In this note we use Laplace transforms and sweep-out sequence to obtain another necessary and sufficient condition.
Defining the sweet-out sequence of A $$\tilde{s}_A(k):=\mu(A^c\cap\cdots\cap T^{-k+1}A^c), $$
\begin{theorem}
The limit distribution of the normalized return (or hitting) time $\mu(A_n)\tau_n$ is exponential if and only if for every $t>0$
\begin{equation}\label{eq:swoseries}\lim_{n\to\infty}(1-e^{-\mu(A_n)t})\sum_{k=0}^\infty e^{-\mu(A_n)kt}\tilde{s}_{A_n}(k)=\frac{t}{t+1}.
\end{equation}
\end{theorem}
Theorem \ref{thm:equiv} details the statement of this result. In particular it provides another way proving the limit distribution of the normalized return time is exponential if and only if the limit distribution of the normalized hitting time is exponential. This method can be used to deal with more general case of stopping times as shown in the last section.  

Generally speaking the sweep-out sequence and the power series in \eqref{eq:swoseries} are not easy to estimate, we have not yet been able to find new systems satisfying exponential limit distribution via this criterion. Nevertheless this power series suggests us to investigate the exponential decay of $\tilde{s}_A(k)$, namely
\begin{equation}\label{def:esr}
\rho_A:=\lim_{k\to\infty}-\frac{\log\tilde{s}_A(k)}{k}
\end{equation}
whenever this limit exists. It is called escape rate by many authors, for example \cite{by2011,fp2012,kl2009,k2012}. An heuristic argument of continuing \eqref{eq:swoseries} analytically to complex $t$ with $\Re t>-1$ hints at
\begin{equation}\label{eq:limesr}
\lim_{n\to\infty}\frac{\rho_{A_n}}{\mu(A_n)}=1,
\end{equation}
although it seems unlikely \eqref{eq:limesr} is in general a necessary or a sufficient condition of \eqref{eq:swoseries}. In fact such a limit in \eqref{eq:limesr} has been studied  in \cite{by2011} for full shifts, in \cite{fp2012} for conformal repellers and in \cite{kl2009,k2012} for systems with Rare-Event-Perron-Frobenius-Operators defined there. When \eqref{eq:limesr} holds for these examples, the limit distribution of the normalized return time turns out to be exponential, in other words \eqref{eq:swoseries} is true. We note that most of the systems been studied possess some kind of uniform spectral gap properties for transfer operators. In Theorem \ref{thm:rholim} below we prove \eqref{eq:limesr} by direct calculations in general $\psi$-mixing systems for a class of targets (Definition \ref{def:cls}). Additionally the example we calculate for does satisfy \eqref{eq:swoseries} as well. Different from methods employed in the aforementioned results, our calculation does not involve transfer operators despite that it depends heavily on the $\psi$-mixing property. 

\section{Sweep-out sequence and limit distribution of normalized return times}\label{sec:swo}
Let $(X,\mathcal B,\mu)$ be a probability space and let $T$ be a measure-preserving and ergodic transformation. For a set $A$ with positive measure, the return time (and the hitting time) map $\tau_A$ is defined as in the \eqref{def:return}. Whether $\tau_A$ means return time or hitting time will be clear by context. 
\begin{definition}
We define the sweep-out sequence $\{\tilde{s}_A(k)\}_{k=0}^\infty$ of $A$ to be 
$$\tilde{s}_A(k):=\mu(A^c\cap\cdots\cap T^{-k+1}A^c) ~\text{ for every } k\ge 1$$ 
and $\tilde{s}_A(0):=1$.
\end{definition}
It is clear that $\{\tilde{s}_A(k)\}_{k=0}^\infty$ is a decreasing sequence. A set $A$ is called a sweep-out set in \cite{dgs1976} if $\lim_{k\to\infty}\tilde{s}_A(k)=0$. Notice that the system is ergodic if and only if every positively measured set is a sweep-out set. Hence for a fixed $A$, $\{\tilde{s}_A(k)\}_{k=0}^{\infty}$ is a sequence decreasing from $1$ to $0$. Also it is clear that $\tilde{s}_{A}(k)$ is bounded from below by $\max\{1-k\mu(A),0\}$. The bound $1-k\mu(A)$ is achieved when $A,\ldots,T^{-k+1}A$ are pairwise disjoint. Since $\mu$ is invariant by $T$, we can write 
\begin{align*}\tilde{s}_A(k)&=\mu(A^c\cap\ldots\cap T^{-k+1}A^c)=\mu(T^{-1}(A^c\cap\cdots\cap T^{-k+1}A^c))\\&=\mu(\tau_A>k).\end{align*}
\begin{remark} 
 In infinite ergodic theory $\{s_A(k)\}$ is named wandering rate. It can be used to describe return sequences for rationally ergodic transformations, see for example \cite{a1997}. 
\end{remark}

We briefly recall the results from \cite{hsv1999} and \cite{hlv2005}. There the difference between limit distribution of $\mu(A)\tau_A$ and the exponential distribution is estimated by the difference $\{c_A(k)\}$ between (un-normalized) hitting time distribution and return time distribution, $$c_A(k):=\mu(\tau_A>k)-\mu_A(\tau_A>k).$$
Let $c_A :=\sup_{k\ge0}|c_A(k)|$ and $\tilde{c}_A :=\sup_{t\ge0}|\mu_A(\mu(A)\tau_A>t)-e^{-t}|$. Theorem 2.1 in \cite{hsv1999} states that
\begin{equation*}
\tilde{c}_A\le 4\mu(A)+c_A(1-\log c_A)
\text{ and }c_A\le 2\mu(A)+ \tilde{c}_A(2-\log\tilde{c}_A).
\end{equation*}
Therefore when $\mu(A)$ shrinks to $0$, $\tilde{c}_A$ converges to $0$ if and only if $c_A$ converges to $0$. That is to say the limit distribution of the normalized return time is exponential if and only if the limit of $c_A$ is $0$. A result in \cite{hlv2005} is that the limit distribution of the normalized return time is exponential if and only if the limit distribution of the normalized hitting time is also exponential. The latter half rephrased in symbols is  $\sup_{t\ge0}|\mu(\mu(A)\tau_A>t)-e^{-t}|\to 0$. Using the notation of sweep-out sequence, it is equivalent to
\begin{equation}\label{eq:equiv1}
\lim_{\mu(A)\to0}\sup_k|\tilde{s}_A(k)-e^{-\mu(A)k}|=0.\end{equation}
Observe that $c_A$(k) can be represented by sweep-out sequence.
\begin{lemma}\label{lemma:sc}
\begin{equation*}\tilde{s}_A(k)-(1-\mu(A))\tilde{s}_A(k-1)=\mu(A) c_A(k-1).\end{equation*}
\end{lemma}
\begin{proof}
By definition we have 
\begin{align*}\tilde{s}_A(k)&=\mu(A^c\cap\cdots\cap T^{-k+1}A^c)\\
&=\mu(T^{-1}A^c\cap\cdots\cap T^{-k+1}A^c)-\mu(A\cap T^{-1}A^c\cap\cdots\cap T^{-k+1}A^c)\\
&=\mu(\tau_A>k-1)-\mu(A)\mu_A(\tau_A>k-1)\\
&=(1-\mu(A))\mu(\tau_A>k-1)+\mu(A)c_A(k-1)\\
&=(1-\mu(A))\tilde{s}_A(k-1)+\mu(A)c_A(k-1).
\end{align*}
\end{proof}
\begin{remark}
A recursive sum of this identity implies $$\tilde{s}_A(k)=(1-\mu(A))^{k}+\mu(A)\sum_{j=0}^{k-1}(1-\mu(A))^{k-1-j}c_A(j).$$
Another identity contained in the proof is 
\begin{equation}\label{eq:so}
\tilde{s}_A(k-1)-\tilde{s}_A(k)=\mu(A)\mu_A(\tau_A>k-1).
\end{equation}
\end{remark}

Having collected the results from \cite{hsv1999,hlv2005} along with Lemma \ref{lemma:sc}, we have
\begin{corollary}
As  $\mu(A)\to0$, the distribution of $\mu(A)\tau_A$ converges to the exponential distribution with parameter $1$ if and only if
\begin{equation}\label{eq:equiv2}\tilde{s}_A(k)-(1-\mu(A))\tilde{s}_A(k-1)=o(\mu(A)), \text{uniformly in } k.\end{equation}
\end{corollary}

Both \eqref{eq:equiv1} and \eqref{eq:equiv2} are sufficient and necessary conditions for the normalized return time converging in distribution to the exponential distribution. Instead of estimating the difference between distribution functions, we go after the weak limit of the normalized return time by Laplace transforms as in \cite{h1993} and present another sufficient and necessary condition. Recall that the Laplace transform of the normalized return time $\mu(A)\tau_A$ on $(A, \mu_{A})$ is $$ \phi_{A}(t)=\int_{A} e^{-\mu(A)\tau_A t}d\mu_{A}, \text{ for every } t>0,$$
and the Laplace transform of the normalized hitting time $\mu(A)\tau_A$ on $(X,\mu)$ is $$ \phi_{A, X}(t)=\int_X e^{-\mu(A)\tau_A t}d\mu, \text{ for every } t>0.$$
We rewrite the Laplace transforms using the sweep-out sequence $\tilde{s}_A(k)$.
\begin{lemma}\label{lem:laplace}
\begin{gather*}
\phi_{A}(t)=1-\frac{e^{\mu(A)t}-1}{\mu(A)}+\frac{e^{\mu(A)t}-1}{\mu(A)}(1-e^{-\mu(A)t})\sum_{k=0}^\infty e^{-\mu(A)kt}\tilde{s}_A(k),\\
\phi_{A,X}(t)=1-(1-e^{-\mu(A)t})\sum_{k=0}^{\infty}e^{-\mu(A)kt}\tilde{s}_A(k).
\end{gather*}
\end{lemma}
\begin{proof}
Note that \eqref{eq:so} implies $$\mu_{A}(\tau_A=k)=\frac{1}{\mu(A)}[(\tilde{s}_A(k-1)-\tilde{s}_A(k))-(\tilde{s}_A(k)-\tilde{s}_A(k+1))].$$
Hence \begin{align*}
\phi_A(t)&=\sum_{k=1}^\infty e^{-\mu(A)kt}\mu_A(\tau_A=k)\\
&=\frac{1}{\mu(A)}\sum_{k=1}^\infty e^{-\mu(A)kt}[(\tilde{s}_A(k-1)-\tilde{s}_A(k))-(\tilde{s}_A(k)-\tilde{s}_A(k+1))]\\
&=1-\frac{1-e^{-\mu(A)t}}{\mu(A)}\sum_{k=0}^\infty e^{-\mu(A)kt}(\tilde{s}_A(k)-\tilde{s}_A(k+1))\\
&=1-\frac{e^{\mu(A)t}-1}{\mu(A)}+\frac{e^{\mu(A)t}-1}{\mu(A)}(1-e^{-\mu(A)t})\sum_{k=0}^\infty e^{-\mu(A)kt}\tilde{s}_A(k).
\end{align*}
For the other one the obvious equality $$\mu(\tau_A=k)=\tilde{s}_A(k-1)-\tilde{s}_A(k)$$ is used to obtain
\begin{align*}
\phi_{A,X}(t)&=\sum_{k=1}^\infty e^{-\mu(A)kt}\mu(\tau_A=k)\\
&=\sum_{k=1}^\infty e^{-\mu(A)kt}(\tilde{s}_A(k-1)-\tilde{s}_A(k))\\
&=1-(1-e^{-\mu(A)t})\sum_{k=0}^{\infty}e^{-\mu(A)kt}\tilde{s}_A(k).
\end{align*}
\end{proof}

Our main theorem in this section is an immediate consequence. 
\begin{theorem} \label{thm:equiv} Let $(X,\mathcal B,\mu, T)$ be a probability-preserving and ergodic dynamical system. For every sequence of positively-measured sets $\{A_n\}$ with $\lim\limits_{n\to\infty}\mu(A_n)=0$,
the following statements are equivalent.
\begin{enumerate}
\item The distribution of $\mu(A_n)\tau_{A_n}$ on $(A_n,\mu_{A_n})$ converges to the exponential distribution with parameter $1$.
\item The distribution of $\mu(A_n)\tau_{A_n}$ on $(X,\mu)$ converges to the exponential distribution with parameter $1$.
\item For every $t>0$ \begin{equation}\label{eq:solaplace}\lim_{n\to\infty}(1-e^{-\mu(A_n)t})\sum_{k=0}^\infty e^{-\mu(A_n)kt}\tilde{s}_{A_n}(k)=\frac{t}{t+1}.
\end{equation}
\end{enumerate}
\end{theorem}
\begin{proof}
Provided the existence of the concerned limits, the preceding lemma implies 
\begin{gather*}\lim_{n\to\infty}\phi_{A_n}(t)=1-t+t\cdot\lim_{n\to\infty}(1-e^{-\mu(A_n)t})\sum_{k=0}^\infty e^{-\mu(A_n)kt}\tilde{s}_{A_n}(k),\\
\lim_{n\to\infty}\phi_{A_n, X}(t)=1-\lim_{n\to\infty}(1-e^{-\mu(A_n)t})\sum_{k=0}^\infty e^{-\mu(A_n)kt}\tilde{s}_{A_n}(k).\end{gather*}
Therefore $\phi_{A_n}(t)$ or $\phi_{A_n, X}(t)$ converges to $\frac1{t+1}$ if and only if \eqref{eq:solaplace} holds. Since the Laplace transform of the exponential distribution with parameter $1$ is exactly $\frac1{t+1}$, which is continuous at $t=0$, the claimed equivalence follows from the continuity theorem for Laplace transforms.
\end{proof}

%\begin{remark}\label{rmk:sum}
%Note that $\tilde{s}_{A_n}(k)=\mu(\tau_{A_n}>k+1)=\mu(S_k=0)$, where $S_k=1_{A_n}+1_{A_n}\circ T+\ldots+1_{A_n}\circ T^k$. This may remind one of local limit theorem. But the exponential law in our hope here differs from the normal density in local limit theorem. {!!!Recheck!!!} A closer look shows that $0$ does not satisfy the convergence condition needed for the  classic local limit theorem, when $S_k$ is replaced by the partial sum of i.i.d. Bernoulli random variables. This limit distribution of return times suggests some properties complementary to local limit theorem. 
%\end{remark}

In practice neither sweep-out sequence itself nor the power series in \eqref{eq:solaplace} is straightforward to estimate. Compared with the other necessary and sufficient conditions \eqref{eq:equiv1} and \eqref{eq:equiv2}, calculation-wise \eqref{eq:solaplace} may or may not have the edge over them. Nonetheless as discussed in the introduction the condition \eqref{eq:solaplace} hints at some phenomenon of escape rates which may characterize when do the normalized return times converge weakly to exponential distribution. As an example we calculate the sweep-out sequence for full shifts borrowing the tools from \cite{go1978,go1981}. Similar calculations have already appeared in \cite{by2011}.

\begin{example} Consider a full shift of $q$ symbols $(\Sigma^{\mathbb N}, T, \mu)$,  $q\ge 2$. $\mu$ is the usual uniform measure, in the sense that if $A$ is a cylinder set $[a_1\ldots a_l]$ then $\mu(A)=q^{-l}$. Introducing $f_A(n)$ as the number of strings of length $n$ that do not contain $A$, then 
\begin{equation}\label{eq:swoshift}\tilde{s}_A(k)=\mu(A^c\cap\ldots\cap T^{-k+1}A^c)=q^{-k-l}\cdot f_A(k+l).
\end{equation} According to \cite{go1981} Theorem 1.1 the generating function $F_A(z)=\sum_{n=0}^\infty f_A(n)z^{-n}$ is a rational function:$$F_A(z)=\frac{z\cdot h_A(z)}{1+(z-q)h_A(z)}$$ where $h_A(z)$ is the auto-correlation function of $A$ as defined in \cite{go1981}. It is sufficient for our purpose to know that $h_A(z)$ is a polynomial in $z$ with coefficients $0$ or $1$ and of degree $l-1$, $l$ being the length of $A$. Lemmas 3, 4 and 5 in \cite{go1978} show that there is a constant $c$ depending only on $q$ such that if $l\ge c$ then $1+(z-q)h_A(z)$ has exactly one root $r_A$ with $|r_A|\ge1.7$. %clearly $|r|<q$, 
The root can be expanded as: \begin{equation}
r_A=q-\frac{1}{h_A(q)}-\frac{h_A'(q)}{h_A^3(q)}+O(\frac{l^2}{q^{3l}}),\label{eq:root}\end{equation}
whence \begin{equation}\label{eq:pattern}f_A(n)=\frac{r_A^n}{1-(r_A-q)^2h_A'(r)}+O(1.7^n).\end{equation}
The constants indicated by $O(\cdot)$ depend on $q$ but are independent of $n$ or $A$, i.e. $h_A$ or $l$. Therefore by \eqref{eq:pattern} the escape rate 
\begin{align}\rho_A&:=-\lim_{k\to\infty}\frac{1}{k}\log\tilde{s}_A(k)=-\lim_{k\to\infty}\frac{1}{k}\log (q^{-k-l}\cdot f_A(k+l))\notag\\
&=-\lim_{k\to\infty}\frac{1}{k}\log \left(\frac{1}{1-(r_A-q)^2h_A'(r)}(\frac{r_A}{q})^{k+l}+O((\frac{1.7}{q})^{k+l})\right)\notag\\&=-\log r_A+\log q.\label{eq:erfs}\end{align}
We estimate $\tilde{s}_A(k)$ against $e^{-k\rho_A}$,
$$\begin{aligned}\tilde{s}_A(k)-e^{-k\rho_A}&=\frac{1}{1-(r_A-q)^2h_A'(r_A)}(\frac{r_A}{q})^{k+l}+O((\frac{1.7}{q})^{k+l})-(\frac{r_A}{q})^k\\
&=\left(\frac{1}{1-(r_A-q)^2h_A'(r_A)}(\frac{r_A}{q})^{l}-1\right)(\frac{r_A}{q})^k+(\frac{1.7}{q})^{l} O((\frac{1.7}{q})^{k}).\end{aligned}$$
Use $\eqref{eq:root}$ to estimate $$\frac{1}{1-(r_A-q)^2h_A'(r_A)}(\frac{r_A}{q})^{l}=\frac{1}{1-(\frac{r_A}{q})^{l-2}O(\frac{l}{q^{l}})}(\frac{r_A}{q})^{l}$$
and $$(\frac{r_A}{q})^{l}=(1-O(\frac{1}{q^{l}}))^{l}.$$

Consider a sequence of cylinders $A_n$ shrinking to a single point with their lengths $l_{A_n}$ growing to infinity. We deduce from previous estimates that 
\begin{equation}\label{eq:swoesr}
\lim_{n\to\infty}\sup_k|\tilde{s}_{A_n}(k)-e^{-k\rho_{A_n}}|= 0.
\end{equation} Thus in this scenario whether the normalized return times $\mu(A_n)\tau_n=q^{-l_{A_n}}\tau_n$ converges in distribution to the exponential distribution is a matter of the escape rates. As for escape rates, \eqref{eq:root} and \eqref{eq:erfs} imply
$$\lim\limits_{n\to\infty}\frac{\rho_{A_n}}{\mu(A_n)}=\lim\limits_{n\to\infty}\frac{q^{l_{A_n}}}{qh_{A_n}(q)}.$$ By investigating the auto-correlation function, it is easy to show that this limit is either $1$ when  $A_n$ shrinks to an aperiodic point or is $1+q^{-m}$ when $A_n$ shrinks to an $m$-periodic point. In view of \eqref{eq:swoesr}, this limit is forced equal to $1$ if the limit distribution of normalized return times is exponential, i.e. if \eqref{eq:solaplace} holds.
%Now consider a new state space $\Sigma'=\Sigma^l\setminus\{a_1\ldots a_l\}$. A state in $\Sigma'$ is an $l$-string in $\Sigma^l$. Define an index set $I$ to order elements of $\Sigma'$, $\Sigma'=\{s_i:i\in I\}$, and a $\{0,1\}$-transition matrix $M=\{m_{ij}\}_{i,j\in I}$ by $m_{ij}=1$ if ${s_i}_{2}\ldots{s_i}_l={s_j}_1\ldots{s_j}_{l-1}$ or $m_{ij}=0$ otherwise. Now we have formulated a subshift $(\Sigma'^{\mathbb N}_M, T', \mu')$ and $\tilde{s}_A(n)=m^{-n-l}\sum_{i,j\in I}(M^n)_{ij}$. This matrix $M$ may be reducible.
\end{example}

\section{Escape rates in $\psi$-mixing systems}
With Theorem \ref{thm:equiv} in mind, we now restrict ourselves in $\psi$-mixing systems and move toward \eqref{eq:limesr}. A common tool to study escape rates is the Perron-Frobenius transfer operator, since an escape rate is just the logarithm of the largest eigenvalue of a perturbation of the transfer operator when a spectral gap is present. Our approach exploits the $\psi$-mixing property alone, leaving out a need for spectral gaps.

Suppose $(X, \mathcal B, T, \mu)$ is a probability-preserving and ergodic system with a filtration $\{\mathcal B_n\}_{n\in\mathbb N}$ of $\mathcal B$, where $\mathcal B_n\subset\mathcal B$ and $T^{-1}\mathcal B_n\subset \mathcal B_{n+1}$ for every $n$. Take a set $A\in\mathcal B$. To shorten notations sometimes we use
\begin{equation*}\tilde{A}^k:=A^c\cap\cdots\cap T^{-k+1}A^c ~\text{ and }~ A^k:=A\cup\cdots\cup T^{-k+1}A.\end{equation*} Assume that $\mu$ has the following mixing property with respect to $\{\mathcal B_n\}_{n\in\mathbb N}$.
\begin{definition}[$\psi$-mixing]
There exists $\psi: \mathbb N\to\mathbb R_+$ such that 
\begin{enumerate}
\item $\psi_n\to 0$ as $n\to\infty$, and
\item for every $n,m,k\in\mathbb N$, any $A\in\mathcal B_n$ and $B\in\mathcal B_m$,
$$|\mu(A\cap T^{-k-n} B)-\mu(A)\mu(B)|\le \psi_k\mu(A)\mu(B).$$
To simplify our calculation we also assume that 
\item $\psi_1=\max_{n\in\mathbb N}\{\psi(n)\}<\infty$.
\end{enumerate}
\end{definition}

We attempt to figure out the asymptotic decay rate of the sweep-out sequence of $A$ as $A$ evolves. Let $n_A$ be the least integer such that $A\in\mathcal B_{n_A}$. The $\psi$-mixing property implies that the sweep out sequence decays exponentially. 

\begin{proposition}\label{prop:exist-rho}
For a $\psi$-mixing system as above and any positively-measured $A\in\mathcal B$ such that $\tilde{s}_A(k)>0$ for all $k\in\mathbb N$, the escape rate $$-\lim_{k\to\infty}\frac{\log \tilde{s}_A(k)}{k}:=\rho_A$$ exists. 
\end{proposition}

We omit the proof to this proposition for it is well-known as a direct outcome of the following (almost) sub-additive property.
\begin{lemma}
For any $m, k\in\mathbb N$, 
\begin{eqnarray}
\log\tilde{s}_A(m+k+n_A)\le \log\tilde{s}_A(m)+\log\tilde{s}_A(k)+\log(1+ \psi_1).\label{eq:psi-sq}
\end{eqnarray}
\end{lemma}
\begin{proof}
Because $$\tilde{s}_A(m+k+n_A)=\mu(\tilde{A}^{-m-k-n_A})\le\mu(\tilde{A}^m\cap T^{-m-n_A}\tilde{A}^k)$$
and because $\tilde{A}^m=A^c\cap\cdots\cap T^{-m+1}A^c\in \mathcal B_{n_A+m-1}$, \eqref{eq:psi-sq} follows from the $\psi$-mixing property.
\end{proof}

Clearly $\rho_A$ takes its value within $[0,+\infty]$. We show that for $A$ varying inside a certain class of sets $\rho_A$ is asymptotically $\mu(A)$. 
\begin{definition}\label{def:cls} For any $\epsilon>0$ consider $A\in\mathcal B$ for which we can find an integer $\ell_A$ such that
\begin{enumerate}
\item $\psi_{\ell_A}<\epsilon$,
\item $(n_A+\ell_A)\cdot\mu(A)<\epsilon$,
\item $n_A\cdot \sup_{1\le i\le n_A}\frac{\mu(A\cap T^{-i}A)}{\mu(A)}<\epsilon.$
\end{enumerate} 
Let $r_A:=n_A+\ell_A$. We denote the family of all such sets by $\mathcal A_\epsilon$.
\end{definition}
\begin{remark}
Examples can be found in \cite{de1995}.
\end{remark}
\begin{theorem}\label{thm:rholim}
If $A_n\in\mathcal A_{\frac 1n}$, then \begin{equation}\label{eq:limesrpsi}\lim_{n\to\infty}\frac{\rho_{A_n}}{\mu(A_n)}= 1.\end{equation}
\end{theorem}
The proof will be built on the next two lemmas. We begin with an estimate of the upper bound.
\begin{lemma}\label{lem:rhoupper} For $\epsilon<0.1$ and any $A\in\mathcal A_\epsilon$ let 
$$q_A:=1-r_A\mu(A)\left(1+\psi_{\ell_A}+2r_A\mu(A)\right),$$ then
$$\rho_A\le -\frac 1 {r_A}\log q_A.$$
\end{lemma}
\begin{proof}
 Note that $q_A$ is chosen as to satisfy 
$$1-q_A^{-1}r_A\mu(A)(1+\psi_{\ell_A})\ge q_A.$$ This can be easily checked using $A\in\mathcal A_\epsilon$. For any $m\ge 1$, we are about to show by induction 
\begin{equation}\label{eq:upperexp}
\tilde{s}_A(mr_A)\ge\tilde{s}_A((m-1)r_A)\cdot q_A.
\end{equation}
When $m=1$,
$\tilde{s}_A(r_A)\ge 1-r_A\mu(A)\ge q_A.$ Assume $\tilde{s}_A(mr_A)\ge q_A\cdot \tilde{s}_n((m-1)r_A)$ holds for some $m\ge 1$, then
\begin{align*}
\tilde{s}_A((m+1)r_A)&=\tilde{s}_A(mr_A)-\mu(\tilde{A}^{mr_A}\cap T^{-mr_A}A^{r_A})\\
&\ge \tilde{s}_A(mr_A)-\mu(\tilde{A}^{(m-1)r_A}\cap T^{-mr_A}A^{r_A})\\
&\ge \tilde{s}_A(mr_A)-\tilde{s}_A((m-1)r_A)\mu(A^{r_A})(1+\psi_{\ell_A})\\
&\ge \tilde{s}_A(mr_A)- q_A^{-1}\tilde{s}_A(mr_A)r_A\mu(A)(1+\psi_{\ell_A})\\
&=\tilde{s}_A(mr_A)(1-q_A^{-1}r_A\mu(A)(1+\psi_{\ell_A}))\\
&\ge \tilde{s}_A(mr_A)\cdot q_A.
\end{align*}
Therefore for any $m\ge 1$,
$\tilde{s}_A(mr_A)\ge q_A^m$ and $\rho_A\le -\frac{1}{r_A}\log q_A$ follows.
\end{proof}

\begin{lemma}\label{lem:rholower}
For $\epsilon<0.1$, any $A\in\mathcal A_\epsilon$ and integer $1\le k<\frac 1\epsilon$,
 $$\rho_A\ge -\frac{\log\left[1-kr_A\mu(A)\left(1-\epsilon-k\epsilon(1+\psi_1)\right)(1-\psi_{\ell_A})\right]}{(k+1)r_A}.$$
\end{lemma}
\begin{proof}
First we note that for every $m, n\ge 1$,
\begin{align*}
\tilde{s}_A(m+n+r_A)&=\tilde{s}_A(m)-\mu(\tilde{A}^m\cap T^{-m}A^{n+r_A})\\
&\le \tilde{s}_A(m)-\mu(\tilde{A}^m\cap T^{-m-r_A}A^{n})\\
&\le \tilde{s}_A(m)-\tilde{s}_A(m)\mu(A^n)(1-\psi_{\ell_A}).
\end{align*}
Hence for every $m\ge 2, k\ge1$,
$$\tilde{s}_A(m(k+1)r_A)\le \tilde{s}_A((m-1)(k+1)r_A)\left(1-\mu(A^{kr_A})(1-\psi_{\ell_A})\right).$$
Let $\hat{p}_A(k):=1-\mu(A^{kr_A})(1-\psi_{\ell_A})$, apply the above inequality recursively 
\begin{equation}\label{eq:lowerexp}
\tilde{s}_A(m(k+1)r_A)\le \tilde{s}_A((k+1)r_A) \hat{p}_A(k)^{m-1}.
\end{equation}
Reformulate this inequality slightly
$$-\frac{\log \tilde{s}_A(m(k+1)r_A)}{m(k+1)r_A} \ge -\frac{\log \tilde{s}_A((k+1)r_A)}{m(k+1)r_A}-\frac{m-1}{m(k+1)r_A}\log \hat{p}_A(k).$$
Hence as $m\to\infty$, for every $k\ge 1$, $$\rho_A\ge -\frac{\log \hat{p}_A(k)}{(k+1)r_A}=-\frac{\log [1-\mu(A^{kr_A})(1-\psi_{\ell_A})]}{(k+1)r_A}.$$
It suffices to estimate $\mu(A^{kr_A}).$ The inclusion-exclusion principle implies
\begin{align*}
\mu(A^{kr_A})&\ge kr_A\mu(A)-\sum_{0\le i<j\le kr_A-1}\mu(T^{-i}A\cap T^{-j}A)\\
&=kr_A\mu(A)-\sum_{1\le i\le kr_A-1}(kr_A-i)\mu(A\cap T^{-i}A)\\
&=kr_A\mu(A)-(\sum_{1\le i\le n_A}+\sum_{n_A+1\le i\le kr_A-1})(kr_A-i)\mu(A\cap T^{-i}A)
\end{align*}
Due to our assumption of $A\in\mathcal A_\epsilon$, $n_A\cdot \sup_{1\le i\le n_A}\frac{\mu(A\cap T^{-i}A)}{\mu(A)}<\epsilon$ and $r_A\mu(A)<\epsilon$. Also the $\psi$-mixing property implies for $i\in[n_A+1, kr_A-1]$ that $\mu(A\cap T^{-i}A)\le \mu(A)^2(1+\psi_{i-n_A})$. Therefore we can continue the above estimate as follows.
\begin{align*}
\mu(A^{kr_A})&\ge kr_A\mu(A)-  \frac{n_A(2kr_A-1-n_A)}{2}\sup_{1\le i\le n_A}{\mu(A\cap T^{-i}A)}\\
&\qquad -\frac{(kr_A-n_A-1)(kr_A-n_A)}{2}\mu(A)^2(1+\psi_1)\\
&\ge kr_A\mu(A) - kr_A\mu(A) n_A \sup_{1\le i\le n_A}\frac{\mu(A\cap T^{-i}A)}{\mu(A)}\\
&\qquad - k^2r_A^2\mu(A)^2(1+\psi_1)\\
&\ge kr_A\mu(A)(1-\epsilon-k\epsilon(1+\psi_1)).
\end{align*}
\end{proof}

\begin{proof}[Proof of Theorem \ref{thm:rholim}]
With the help of $-\log(1-x)<(1+\epsilon)x$ for $0\le x<\epsilon/2$ and of $-\log(1-x)>x$ for all $0\le x<1$, the previous two lemmas lead to the theorem. Indeed, suppose $A_n\in\mathcal A_{\frac1n}$ and fix any $\varepsilon>0$. Then for $n>4/\varepsilon$,
$r_{A_n}\mu(A_n)(1+\psi_{\ell_{A_n}}+2r_{A_n}\mu(A_n))<\varepsilon/2$ and hence by Lemma \ref{lem:rhoupper}
\begin{align*}\frac{\rho_{A_n}}{\mu(A_n)}&\le \frac{1}{r_{A_n}\mu(A_n)}(1+\varepsilon) r_{A_n}\mu(A_n)(1+\psi_{\ell_{A_n}}+2r_{A_n}\mu(A_n))\\
&=(1+\varepsilon)(1+\psi_{\ell_{A_n}}+2r_{A_n}\mu(A_n))\le (1+\varepsilon)^2.
\end{align*}
On the other hand if we choose $k=[\sqrt{n}]$ in Lemma \ref{lem:rholower}, as $n\to\infty$
\begin{equation*}
\frac{\rho_{A_n}}{\mu(A_n)}\ge \frac{\sqrt{n}}{\sqrt{n}+1}\left(1-\frac{1}{n}-\frac {\sqrt{n}}{n}(1+\psi_1)\right)(1-\frac 1n)\ge 1-\varepsilon.
\end{equation*}
\end{proof}

\begin{remark}
The estimates \eqref{eq:upperexp} and \eqref{eq:lowerexp} strictly control how the sweep-out sequence decays at every step. One can infer that the same relation as \eqref{eq:swoesr} holds. Hence the necessary and sufficient condition \eqref{eq:solaplace} for exponential limit distribution of normalized return times is checked by \eqref{eq:limesrpsi}. 
\end{remark}

\section{Generalized hitting times}
We return to the settings in section \ref{sec:swo}. Our arguments there can be easily applied to more general hitting times. In fact writing the sweep-out sequence $\tilde{s}_{A}(k)$ as $\mu(\{1_{A}+1_{A}\circ T+\ldots+1_{A}\circ T^{k-1}=0\})$ suggests an apparent way. Henceforth we consider functions $f\in L^1(X,\mu)$ with $0\le f\le 1$ and $\|f\|_1>0$  and define $$\tau_f(x):=\inf\{k\ge 1: f\circ T+\cdots+f\circ T^k\ge1\}.$$
By Birkhoff's ergodic theorem, $\tau_f(x)$ is finite almost everywhere. We are interested in weak limit of $\mu(\supp f)\tau_f$ on the induced system $(\supp f, \mu_{\supp f})$. Studying this more general hitting times may benefit studying the usual return times by approximations. For simplicity sometimes we denote $\supp f$ by $A_f$ and  when $\mu(\supp f)$ is used as a scaling factor we denote it by $\epsilon_f$. 

Recall that the Laplace transforms of $\epsilon_f\tau_f$ on $(A_f, \mu_{A_f})$ and of $\epsilon_f\tau_f$ on $(X,\mu)$ are $ \phi_{f, A_f}(t):=\int_{A_f} e^{-\epsilon_f\tau_f t}d\mu_{A_f}$ and $ \phi_{f, X}(t):=\int_X e^{-\epsilon_f\tau_f t}d\mu$ respectively. Define for every integer $k\ge 0$ $$\tilde{s}_f(k):=\mu(\tau_f>k).$$ Then we have, similar to \eqref{eq:so},
\begin{lemma}
\begin{equation}\label{eq:tauf}\mu_{A_f}(\tau_f>k)=\frac{1}{\mu(A_f)}(\tilde{s}_f(k)-\tilde{s}_f(k+1))+\mu_{T^{-1}A_f}(\tau_f>k+1).\end{equation}
\end{lemma}
\begin{proof} Remember $A_f=\supp f$, then
\begin{align*}
&\quad\mu(A_f)\mu_{A_f}(\tau_f>k)=\mu(\{\tau_f>k\}\cap A_f)\\
&=\mu(\tau_f>k)-\mu(\{\tau_f>k\}\cap A_f^c)\\
&=\mu(\tau_f>k)-\mu(\{f\circ T+\cdots f\circ T^k<1\}\cap A^c_f)\\
&=\mu(\tau_f>k)-\mu(\{f+f\circ T+\cdots f\circ T^k<1\}\cap A^c_f)\\
&=\mu(\tau_f>k)-\mu(T^{-1}(\{f+f\circ T+\cdots f\circ T^k<1\})\cap T^{-1}A_f^c)\\
&=\mu(\tau_f>k)-\mu(\{\tau_f>k+1\}\cap T^{-1}A_f^c)\\
&=\mu(\tau_f>k)-\mu(\tau_f>k+1)+\mu(\{\tau_f>k+1\}\cap T^{-1}A_f)\\
&=\mu(\tau_f>k)-\mu(\tau_f>k+1)+\mu(A_f)\mu_{T^{-1}A_f}(\tau_f>k+1).
\end{align*}
\end{proof}

With this lemma we can rewrite the Laplace transforms using $\{\tilde{s}_f(k)\}_{k=0}^\infty$ similar to Lemma \ref{lem:laplace}.
\begin{lemma}
\begin{equation*}
\phi_{f,A_f}(t)=e^{\epsilon_f t}\phi_{f,T^{-1}A_f}(t)-\frac{e^{\epsilon_f t}-1}{\mu(A_f)}+\frac{e^{\epsilon_f t}-1}{\mu(A_f)}(1-e^{-\epsilon_f t})\sum_{k=0}^\infty e^{-\epsilon_f kt}\tilde{s}_f(k).
\end{equation*}
\end{lemma}
\begin{proof} Substitute \eqref{eq:tauf} into the Laplace transform $\phi_{f,A_f}(t)$ to get
\begin{align*}
\phi_{f,A_f}(t)&=\sum_{k=1}^\infty e^{-\epsilon_f kt}\mu_{A_f}(\tau_f=k)\\
&=\frac{1}{\mu(A_f)}\sum_{k=1}^\infty e^{-\epsilon_f kt}[(\tilde{s}_f(k-1)-\tilde{s}_f(k))-(\tilde{s}_f(k)-\tilde{s}_f(k+1))]\\
&\quad +\sum_{k=1}^\infty e^{-\epsilon_f kt}\mu_{T^{-1}A_f}(\tau_f=k+1)\\
&=\frac{1-\tilde{s}_f(1)}{\mu(A_f)}-\frac{(e^{\epsilon_f t}-1)}{\mu(A_f)}\sum_{k=1}^{\infty} e^{-\epsilon_f kt}(\tilde{s}_f(k-1)-\tilde{s}_f(k))\\
&\quad +e^{\epsilon_f t} \sum_{k=1}^\infty e^{-\epsilon_f kt}\mu_{T^{-1}A_f}(\tau_f=k)-\mu_{T^{-1}A_f}(\tau_f=1).
\end{align*}
Simplify the first summation and use a Laplace transform to contract the second summation to continue
\begin{align*}
\phi_{f,A_f}(t)&=\frac{1-\tilde{s}_f(0)}{\mu(A_f)}-\frac{e^{\epsilon_f t}-1}{\mu(A_f)}+\frac{e^{\epsilon_f t}-1}{\mu(A_f)}(1-e^{-\epsilon_f t})\sum_{k=0}^\infty e^{-\epsilon_f kt}\tilde{s}_f(k)\\
&\quad+e^{\epsilon_f t}\phi_{f, T^{-1}A_f}(t)-\frac{\mu(\{\tau_f=1\}\cap T^{-1}A_f)}{\mu(T^{-1}A_f)}\\
&=\frac{\mu(f=1)}{\mu(A_f)}-\frac{e^{\epsilon_f t}-1}{\mu(A_f)}+\frac{e^{\epsilon_f t}-1}{\mu(A_f)}(1-e^{-\epsilon_f t})\sum_{k=0}^\infty e^{-\epsilon_f kt}\tilde{s}_f(k)\\
&\quad+e^{\epsilon_f t}\phi_{f, T^{-1}A_f}(t)-\frac{\mu(f=1)}{\mu(A_f)}\\
&=e^{\epsilon_f t}\phi_{f,T^{-1}A_f}(t)-\frac{e^{\epsilon_f t}-1}{\mu(A_f)}+\frac{e^{\epsilon_f t}-1}{\mu(A_f)}(1-e^{-\epsilon_f t})\sum_{k=0}^\infty e^{-\epsilon_f kt}\tilde{s}_f(k).
\end{align*}
\end{proof}

\begin{lemma}
\begin{equation*}
\phi_{f,X}(t)=1-(1-e^{-\epsilon_f t})\sum_{k=0}^\infty e^{-\epsilon_f kt}\tilde{s}_f(k).
\end{equation*}
\end{lemma}
\begin{proof}
It can be seen from using $\mu(\tau_f=k)=\tilde{s}_f(k-1)-\tilde{s}_f(k)$, or from replacing $A_f$ by $X$ in the previous lemma as $X\supset A_f$ and $X=T^{-1}X$.
\end{proof}

Regarding a sequence of functions having shrinking supports we conclude the following result.
\begin{theorem}
Suppose $\{f_n\}\subset L^1(X,\mu)$ with $0\le f_n\le 1$ and $\|f_n\|_1>0$ for every $n$ and  suppose that $\epsilon_{n}=\mu(\supp f_n)$ has the limit $0$. If $$\lim\limits_{n\to\infty}\frac{\mu(f_n=1)}{\mu(\supp f_n)}=1,$$ then the following statements are equivalent. 
\begin{enumerate}
\item The distribution of $\epsilon_n\tau_{f_n}$ on $(X,\mu)$ converges to the exponential distribution with parameter $1$.
\item The distribution of $\epsilon_n\tau_{f_n}$ on $(\supp f_n, \mu_{\supp f_n})$ converges to the exponential distribution with parameter $1$.
\item For every $t>0$ $$\lim_{n\to\infty}(1-e^{-\epsilon_{n} t})\sum_{k=0}^\infty e^{-\epsilon_{n} kt}\tilde{s}_{f_n}(k)=\frac{t}{t+1}.$$
\end{enumerate}
\end{theorem}
\begin{proof}
As a result of the above calculations of the Laplace transforms it suffices to show that $\lim\limits_{n\to\infty}\phi_{f_n,T^{-1}\supp f_n}(t)=1$. This is ensured  by $\frac{\mu_{T^{-1}\supp f_n}(\tau_{f_n}=1)}{\mu(T^{-1}\supp f_n)}=\frac{\mu(f_n=1)}{\mu(\supp f_n)}\to 1$ in the assumption.
\end{proof}

\section*{Acknowledgments}
The author was partially supported by NSF Grant DMS-1008538. He thanks Prof. Manfred Denker for many insightful comments and generous help.

%%For cite command type as \cite{1}; \cite{3,6} and \cite{2,4,6}. 
%%For refcite command type as Refs.~[\refcite{1}];   
%%[\refcite{1},\refcite{3}] and [\refcite{1}--\refcite{4}].

\bibliographystyle{plain}
\bibliography{bib}

\end{document}